\documentclass{amsart}

\usepackage{amsmath,amssymb,amsthm}
\usepackage{graphicx,tikz}

\newtheorem*{thm}{Theorem}
\newtheorem*{proposition}{Proposition}
\newtheorem{corollary}{Corollary}

\newtheorem{lemma}{Lemma}

\theoremstyle{definition}

\theoremstyle{remark}

\DeclareMathOperator{\ei}{Ei}

\begin{document}

\title[]{On the logarithmic energy of points on {\Large $\mathbb{S}^2$}}
\thanks{S.S. is supported by the NSF (DMS-2123224) and the Alfred P. Sloan Foundation.}

\author[]{Stefan Steinerberger}
\address{Department of Mathematics, University of Washington, Seattle, WA 98195, USA}
\email{steinerb@uw.edu}

\begin{abstract}
We revisit a classical question: how large is the minimal logarithmic energy of $n$ points on $\mathbb{S}^2$
$$ \mathcal{E}_{\log}(n) = \min_{x_1, \dots, x_n \in \mathbb{S}^2} \quad \sum_{i,j =1 \atop i \neq j}^{n}{ \log{\frac{1}{\|x_i-x_j\|}} }?$$
 B\'etermin \& Sandier (building on work of Sandier \& Serfaty)
showed that
$$ \mathcal{E}_{\log}(n) = \left( \frac{1}{2} - \log{2} \right)n^2 - \frac{n \log{n}}{2} + c_{\log} \cdot n + o(n),$$
where the constant $c_{\log}$ is characterized by a certain renormalized minimization problem. Brauchart, Hardin \& Saff conjectured a closed form expression for $c_{\log}$ ($\sim -0.05$) assuming
analytic continuation.
We describe a simple renormalization approach that results in a purely local problem involving superpositions of Gaussians. In particular, this reaffirms a recent result of Petrache-Serfaty about implications of the Cohn-Kumar conjecture. We also improve the lower bound from $c_{\log} \geq -0.223$ to $c_{\log} \geq -0.095$.

\end{abstract}

\maketitle

\section{Introduction}
We revisit the classical problem of trying to understand the minimal logarithmic energy of a set of points $\left\{x_1, \dots, x_n\right\} \subset \mathbb{S}^2$. This quantity
is defined by 
$$ \mathcal{E}_{\log}(n) = \min_{x_1, \dots, x_n \in \mathbb{S}^2} \quad \sum_{i,j=1 \atop i \neq j}^{n}{ \log{\frac{1}{\|x_i-x_j\|}} },$$
where $\mathbb{S}^2 \subset \mathbb{R}^3$ is the sphere of radius 1 and $\| \cdot \|$ is the Euclidean distance in $\mathbb{R}^3$. A standard potential-theoretic approach
would be to assume that the best case is where the points are fairly evenly distributed over the sphere and that in this case 
$$ \sum_{i,j=1 \atop i \neq j}^n \log{\frac{1}{\|x-y\|}} \sim \frac{n^2}{|\mathbb{S}^2|^2} \int_{\mathbb{S}^2} \int_{\mathbb{S}^2} \log \frac{1}{\|x-y\|} dx dy =  \left( \frac{1}{2} - \log{2} \right)n^2.$$
This is indeed the correct leading order. As for refined estimates, it is known that
   $$ -\frac{n}{2}  \log{n} + c_1 n \leq \mathcal{E}_{\log}  - \left( \frac{1}{2} - \log{2} \right)n^2 \leq -\frac{n}{2}  \log{n} + c_2 n,$$
   where the lower bound follows from a result of Elkies (see Baker \cite{baker} or Lang \cite{lang}) with another proof given by Wagner \cite{wagner}. The upper bound was established, with explicit constants, by Rakhmanov, Saff \& Zhou \cite{rakh}, see \cite{ali, arm} for other examples.  B\'etermin \& Sandier \cite{betermin} adapted a renormalization scheme of Sandier \& Serfaty \cite{sandier} to prove that there indeed exists, asymptotically, a constant in front of the linear term, i.e.
   $$\mathcal{E}_{\log}(n) = \left( \frac{1}{2} - \log{2} \right)n^2 - \frac{n \log{n}}{2} +c_{\log} \cdot n +o(n).$$
   The best known lower bound $c_{\log} \geq -0.223$ is due to Dubickas \cite{dub}. As for the upper bound,  a very nearly optimal explicit construction is due to  Beltran \& Etayo \cite{beltri}. 
    Betermin \& Sandier \cite{sandier} showed that $c_{\log} \leq c_{\tiny \mbox{BHS}}$, where $c_{\tiny \mbox{BHS}}$ denotes a constant earlier introduced by Brauchart, Hardin \& Saff \cite{brauchart}: they interpret the logarithmic energy as the derivative of the $s-$Riesz energy and suggest the following \begin{quote}
   \textbf{Conjecture} (Brauchart, Hardin \& Saff \cite{brauchart})\textbf{.} $c_{\log}$ is equal to $c_{\tiny \mbox{BHS}}$ where
$$  c_{\tiny \mbox{BHS}} =  2\log{2} + \frac{1}{2} \log\frac{2}{3} + 3 \log \frac{\sqrt{\pi}}{\Gamma(1/3)}  \sim -0.055605\dots$$
\end{quote}
We observe that the question has also received attention as being related to Problem 7 on Smale's list \cite{smale}: is it possible to find, with computational cost at most polynomial in $n$, points $\left\{x_1, \dots, x_n\right\} \subset \mathbb{S}^2$ such that
$$ \sum_{i,j=1 \atop i \neq j}^n{ \log{\frac{1}{\|x_i-x_j\|}}} \leq \mathcal{E}_{\log}(n)  + C \log{n}?$$
There is a nice survey by Beltran \cite{beltran00} and, more recently, a detailed description in the book by Borodachov, Hardin \& Saff \cite[Chapter 7]{bor}. We also refer to a very detailed recent description by Hardin, Michaels \& Saff \cite{hardinm}  who survey a large number of popular point configurations on $\mathbb{S}^2$.

 \section{Results}
 \subsection{The Main Result.} The purpose of our paper is to provide a simple renormalization procedure that is completely explicit. Its main ingredient is the heat kernel $e^{t\Delta} \delta_x$ applied to a Dirac measure located in a point $x \in \mathbb{S}^2$. While this object is perhaps somewhat abstract on more general manifolds, there are explicit formulas on $\mathbb{S}^2$ (see e.g. \cite{nagase} and references therein). Moreover, we will only use this object for very small values of $t$ (roughly $t \sim n^{-1}$ as $n \rightarrow \infty$). In that regime, we have Varadhan's short time asymptotics \cite{vara1, vara2} suggesting that if $x, y \in \mathbb{S}^2$ are two points that are very close (and only those will matter), then
 $$ \left[e^{t\Delta} \delta_{x_k}\right](x) \sim \frac{1}{4 \pi t} \exp \left(-\frac{\|x - x_k\|^2}{4t}\right).$$
Therefore, we can think of this object as `basically' a Gaussian when $t$ is small. Moreover, if $t \sim n^{-1}$, then this Gaussian is spaced out at scale $\sim n^{-1/2}$ (which is also the nearest neighbor distance of an extremal configuration: the Voronoi cells should be roughly round with each cell having roughly the same volume $\sim 1/n$). 
 \begin{thm}[Main Result] For all sets of distinct points $\left\{x_1, \dots, x_n\right\} \subset \mathbb{S}^2$ and all $t > 0$, the expression
\begin{align*} \frac{1}{2\pi} \sum_{k, \ell =1 \atop k \neq \ell}^n \log \frac{1}{\|x_k - x_{\ell}\|}  -  \int_{0}^t\sum_{k, \ell =1 \atop k \neq \ell}^n  \left( e^{s \Delta} \delta_{x_{\ell}} \right)(x_k) ds - \left\| e^{(t/2) \Delta} \left( \sum_{k=1}^{n} \delta_{x_k} \right) \right\|_{\dot H^{-1}}^2 
\end{align*}
is \underline{\emph{independent}} of the set of points.
\end{thm}
This means that the problems of minimizing
$$ \min_{x_1, \dots, x_n \in \mathbb{S}^2} \sum_{k, \ell =1 \atop k \neq \ell}^n \log \frac{1}{\|x_k - x_{\ell}\|}$$
and, for any fixed $t>0$,
$$ \min_{x_1, \dots, x_n \in \mathbb{S}^2}  \int_{0}^t \sum_{k, \ell =1 \atop k \neq \ell}^n \left( e^{s \Delta} \delta_{x_{\ell}} \right)(x_k) ds+ \left\| e^{(t/2) \Delta} \left( \sum_{k=1}^{n} \delta_{x_k} \right) \right\|_{\dot H^{-1}}^2  $$
are completely equivalent: their energy differs by a universal constant, depending only on $n$ and $t$, for which we can give fairly precise estimates as $t \rightarrow 0$ (see \S 4.1).
We are completely free to choose $t$ any way we see fit and we will argue in \S 2.2 that $t = c/n$ or maybe $t = (\log{n})^c/n$ for some constant $c \gg 1$ is a particularly natural choice. 
The result also holds on $\mathbb{S}^d$ and certain compact rank one symmetric spaces (basically, what we need is that the heat kernel `looks the same' in every point). 

\subsection{Simplifying the Problem.} We will now analyze the first term in our functional for $s$ fixed, i.e. the problem 
$$\min_{x_1, \dots, x_n \in \mathbb{S}^2} \sum_{k, \ell=1 \atop k \neq \ell}^{n}   \int_{\mathbb{S}^2}\left[ e^{s\Delta} \delta_{x_k} \right](x_{\ell}).$$
If $ s \sim n^{-1}$, then this is, in some sense, a much simpler problem: it is essentially Gaussian interaction at the scale of nearest neighbor distances. This energy functional already arose in a series of other settings \cite{lu, stein0, stein1}. The specific problem that governs the behavior of the minimal logarithmic energy at the linear scale is:
\begin{quote}
\textbf{Problem 1.} Let $0 < c < \infty$ and consider for $\left\{x_1, \dots, x_n\right\} \subset \mathbb{S}^2$
$$ \min_{x_1, \dots, x_n \in \mathbb{S}^2} \sum_{i,j =1}^{n}{ \exp\left( - c \cdot n \cdot \|x_i - x_j\|^2\right)} $$
Prove that minimizing configurations behave locally like the hexagonal lattice (on most of the domain as $n \rightarrow \infty$).
\end{quote}

It is probably not important that the points are on $\mathbb{S}^2$, the statement presumably does not depend much on the underlying manifold. 
This is a classical crystallization conjecture (see e.g. \cite{blanc}). There are many such conjectures for a wide variety of kernels; there is a lot of numerical evidence for most of them (in particular, it is likely that this problem has already been stated many times in the literature). 
The relevance of this problem is captured in the following `Meta'-Theorem. We cannot state it as a Theorem because it depends on how accurately one solves Problem 1.\\

\textbf{`Meta'-Theorem.}  \textit{ A sufficiently quantitative solution of \emph{Problem 1} implies the conjecture of Brauchart, Hardin \& Saff \cite{brauchart}
$$  c_{\log} =  2\log{2} + \frac{1}{2} \log\frac{2}{3} + 3 \log \frac{\sqrt{\pi}}{\Gamma(1/3)}  \sim -0.055605\dots$$
It would also imply similar such statements on other compact rank one symmetric spaces and methods of obtaining the constant.}

\begin{proof}[`Meta'-Proof.] We want to solve, for some $t > 0$,
$$\min_{x_1, \dots, x_n \in \mathbb{S}^2}    \int_{0}^t \sum_{k, \ell =1  \atop k \neq \ell}^n \left( e^{s \Delta} \delta_{x_{\ell}} \right)(x_k) ds+ \left\| e^{(t/2) \Delta} \left( \sum_{k=1}^{n} \delta_{x_k} \right) \right\|_{\dot H^{-1}}^2.$$
By picking $t = c/n$ and using the hypothetical solution to Problem 1, we see that the first term is minimized by a structure that is locally hexagonal. If the points behave locally like a hexagonal lattice and $t = c/n$, then, as $n \rightarrow \infty,$
$$ \left\| e^{(c/(2n)) \Delta} \left( \sum_{k=1}^{n} \delta_{x_k} \right) \right\|_{\dot H^{-1}}^2  \leq f(c)   \int_{0}^{c/n} \sum_{k, \ell = 1 \atop k \neq \ell}^n \left( e^{s \Delta} \delta_{x_{\ell}} \right)(x_k) ds $$
where $f(c) \rightarrow 0$ as $c \rightarrow \infty$.
We give more details in \S 4.4.
\end{proof}

Summarizing, we have the following chain of arguments.
\begin{enumerate}
\item Minimizing logarithmic energy is the same as minimizing, for any $t > 0$,
$$\min_{x_1, \dots, x_n \in \mathbb{S}^2}   \int_{0}^t \sum_{k, \ell =1 \atop k \neq \ell}^n \left( e^{s \Delta} \delta_{x_{\ell}} \right)(x_k) ds+ \left\| e^{(t/2) \Delta} \left( \sum_{k=1}^{n} \delta_{x_k} \right) \right\|_{\dot H^{-1}}^2  $$
\item The first term is an integral over interactions of Gaussian-type. For each $s>0$, we expect the minimizing configuration to be approximately hexagonal and thus also for superpositions and for solutions of
$$    \min_{x_1, \dots, x_n \in \mathbb{S}^2} \int_{0}^t \sum_{k, \ell =1 \atop k \neq \ell}^n \left( e^{s \Delta} \delta_{x_{\ell}} \right)(x_k) ds  $$
which should also be locally hexagonal.
\item If that is the case and we choose $t = c/n$ with $c \gg 1$, then this first term is bigger than the second term. In particular, we could estimate the logarithmic energy by only taking the first term and letting $c \rightarrow \infty$ (regarding order of limits, we would consider a fixed $c$ and then let $n\rightarrow \infty$ and then afterwards remark that we could pick $c$ larger and larger).
\item This would result in the constant predicted by Brauchart, Hardin \& Saff (see \S 5) and establish the linear asymptotics for the logarithmic energy.
\end{enumerate}

This line of reasoning is naturally related to a recent paper of Petrache \& Serfaty \cite{pet}: they showed that the Cohn-Kumar \cite{cohn} conjecture in $d=2$ would imply that the hexagonal lattice is optimal with respect to the renormalized energy (see Sandier \& Serfaty \cite{sandier}) and thus, via the work of B\'etermin \& Sandier \cite{betermin}, would lead to $c_{\log}$ coinciding with the value predicted by Brauchart, Hardin \& Saff \cite{brauchart}. We should thus think of the Gaussian as a basic building block for which these types of crystallization conjectures might be the easiest (though this is hard to say for certain, the Gaussian is certainly a particularly nice function). Establishing such statements for the Gaussian would have several other implications: step (2), writing logarithmic energy as a superposition of Gaussians, is more flexible and would also allow to represent other functions (we refer to Petrache \& Serfaty \cite{pet} for more details).

\subsection{An improved lower bound}  A byproduct of our approach is a new lower bound on $c_{\log}$. We recall that, due to B\'etermin \& Sandier \cite{betermin}, we have the asymptotic expansion
$$\mathcal{E}_{\log}(n) = \left( \frac{1}{2} - \log{2} \right)n^2 - \frac{n \log{n}}{2} +c_{\log} \cdot n +o(n),$$
where $c_{\log}$ is known to satisfy
$$ - 0.223 \leq c_{\log} \leq c_{\tiny \mbox{BHS}} \sim  -0.055605\dots$$
with the lower bound is due to Dubickas \cite{dub} and the upper bound is due to B\'etermin \& Sandier \cite{betermin}. It is widely assumed that the upper bound is sharp.
\begin{corollary} We have
$$ c_{\log} \geq \frac{\log{4} - 1 - \gamma}{2} \sim -0.0954\dots,$$
where $\gamma \sim 0.577\dots$ is the Euler-Mascheroni constant.
\end{corollary}
What is nice about the proof is that it fits into the overarching philosophy: we get this particular value by considering $t = 1/n$ and have to deal with two particular objects, those being
$$    \int_{0}^t \sum_{k, \ell =1 \atop k \neq \ell}^n \left( e^{s \Delta} \delta_{x_{\ell}} \right)(x_k) ds \qquad \mbox{and} \qquad \left\| e^{(t/2) \Delta} \left( \sum_{k=1}^{n} \delta_{x_k} \right) \right\|_{\dot H^{-1}}^2.$$
We bound both from below by 0.  This is fairly accurate for the first term when $0 \leq t \ll 1/n$ and is fairly accurate for the second term when $t \gg 1/n$. The point $t=1/n$ represents a nice middle ground between the two and leads to Corollary 1. One would like to choose $t = c/n$ for $c \gg 1$ to be able to comfortably ignore the second term -- then, however, the first term becomes more important and requires a more detailed understanding of the local geometry of optimal configurations.
(Note added in print: in a recent preprint, Lauritsen \cite{laur} showed that an old bound on the jellium energy due to Lieb-Narnhofer \cite{liebn} and Sari-Merlini \cite{sari} can be used to further improves the estimate to $c_{\log} \geq - 0.0569$.)

\subsection{Cubic Jacobi theta function.}
A byproduct of our argument is an alternative expression for the constant $c_{\log}$ governing the linear term under the assumption of optimal configurations being locally hexagonal. Following Borwein \& Borwein \cite{borwein}, we introduce, for $|q| < 1$, the series
$$ L(q) = \sum_{m,n = -\infty}^{\infty} q^{m^2 + mn + n^2}.$$
This is the cubic analogue of the Jacobi theta function and satisfies a number of interesting identities, we refer to \cite{borwein, markus} for more details.
\begin{corollary} We have
$$ c_{\tiny \emph{BHS}} = \lim_{c \rightarrow \infty}  \int_{0}^{c}  \frac{1}{2s} \left(  L\left(e^{-\frac{2\pi}{\sqrt{3} s}} \right) - 1 \right)ds - \frac{c}{2} + \frac{\log{(4c)}}{2} - \frac{\gamma}{2}.$$
\end{corollary}
We derive the expected energy for a hexagonal lattice using our main result: we will use $t = c/n$ and then let $c$ become large (this is the same $c$ as in the expression in Corollary 2).
Corollary 2 seems to be like it could be related to things that are interesting in their own right. It seems that the convergence happens from below: differentiating the integral in $c$ and asking for the derivative to ultimately nonnegative is the same as asking for
$$  L\left(e^{-\frac{2\pi}{\sqrt{3} c}} \right) \geq c \qquad \mbox{for}~c~\mbox{sufficiently large.}$$
This inequality is equivalent to, for $q$ sufficiently close to 1
$$\sum_{m,n = - \infty}^{\infty}{q^{m^2 + mn + n^2}} \geq \frac{2\pi}{\sqrt{3} \log{(1/q)}}$$
At least numerically, the inequality seems to be true for all $0 < q < 1$. Once $q$ gets close to 1, we can use results from asymptotic analysis to argue as follows:  using an identity of Borwein \& Borwein \cite{borwein}, we have
$$ L(q) = \theta_3(q) \theta_3(q^3) + \theta_2(q) \theta_2(q^3),$$
where
$$ \theta_2(q) = \sum_{k=-\infty}^{\infty} q^{(k+1/2)^2} \qquad \mbox{and} \qquad \theta_3(q) = \sum_{k=-\infty}^{\infty} q^{k^2}$$
are the classical Jacobi theta functions. We thus require results for how they asymptotically behave as $q \rightarrow 1$. Here we refer to a result of Olde Daalhuis \cite{olde}. The special case $z=1$ in \cite[\S 3.15.b]{olde} and \cite[\S 3.14.c]{olde} is
\begin{align*}
\theta_2(q) = (q^2, q^2)_{\infty} \cdot \exp\left( - \frac{1}{\log{q}} \frac{\pi^2}{12} + \frac{\log{q}}{12}+ \sum_{k=1}^{\infty} \frac{1}{k \sinh \left(\frac{\pi^2 k}{\log{q}}\right)} \right),
\end{align*}
and
\begin{align*}
\theta_3(q) = (q^2, q^2)_{\infty} \cdot \exp\left( - \frac{1}{\log{q}} \frac{\pi^2}{12} + \frac{\log{q}}{12}+ \sum_{k=1}^{\infty} \frac{(-1)^k}{k \sinh \left(\frac{\pi^2 k}{\log{q}}\right)} \right),
\end{align*}
where $(a,q)_{\infty}$ is the $q-$analogue of the Pochhammer symbol. We thus require asymptotic results for 
$$(q^2, q^2)_{\infty} = \prod_{k=0}^{\infty} (1- q^{2n+2}).$$
 These are classical (see \cite[Chapter 3]{apostol} or \cite[Corollary 1.2]{kat}) and
$$ (q^2; q^2)_{\infty} = \exp\left(  - \frac{\pi^2}{12 \log{(1/q)}} - \frac{1}{2} \log \left( \frac{\log{(1/q)}}{\pi} \right) + \frac{\log{(1/q)}}{12} - \sum_{k=1}^{\infty} \frac{1}{k} \frac{\widehat{q}^k}{1 - \widehat{q}^k}  \right),$$
where $\widehat{q}$ is an abbreviation for 
$$ \widehat{q} = \exp \left(-\frac{2\pi^2}{\log{(1/q)}}\right).$$
These identities suffice to establish the desired inequality for $q$ close to 1 but one would certainly expect it to be true for all $0 < q < 1$. None of this impacts our main argument but these types of considerations may become useful when trying to use the main identity at very small scales $t \sim (\log{n})^c/n$.

\subsection{Incomplete Gamma Function.}
As another byproduct, we obtain some very precise energy asymptotics for the incomplete gamma function
$$ \Gamma(0,z) = \int_{z}^{\infty} \frac{e^{-t}}{t} dt$$
when evaluated over the hexagonal lattice. The standard hexagonal lattice $\Lambda$ in $\mathbb{R}^2$ is generated by the vectors
$$ v_1 = (1,0) \qquad \mbox{and} \qquad v_2 = \left(\frac{1}{2}, \frac{\sqrt{3}}{2}\right).$$
We note that, for $k_1, k_2 \in \mathbb{Z}$,
$$ \left\| k_1 v_1 + k_2 v_2 \right\|^2 =  k_1^2 + k_1 k_2 + k_2^2.$$
Therefore,
$$ \sum_{\lambda \in \Lambda \atop \lambda \neq 0} \Gamma(0, \varepsilon \| \lambda\|^2)=  \sum_{  (k_1, k_2) \in \mathbb{Z}^2 \atop (k_1, k_2) \neq (0,0)}  \Gamma\left(0,  \varepsilon (k_1^2 + k_1 k_2 + k_2^2) \right).$$
We obtain a three-term expansion.

\begin{corollary}
As $\varepsilon \rightarrow 0$, we have
 \begin{align*}  \sum_{  (k_1, k_2) \in \mathbb{Z}^2 \atop (k_1, k_2) \neq (0,0)}  \Gamma\left(0,  \varepsilon (k_1^2 + k_1 k_2 + k_2^2) \right) &= \frac{2\pi}{\sqrt{3}} \frac{1}{\varepsilon} - \log{\left(\frac{1}{\varepsilon}\right)} 
 + \gamma + \log \left(\frac{4\pi^2}{\sqrt{3}  \Gamma  \left(\frac{1}{3} \right)^6}\right) + o(1).
\end{align*}
\end{corollary}
We observe that $\Gamma(0,z)$ does decay exponentially for $|z| \gtrsim 1$ and it is thus relatively easy to check this expansion (since only few lattice points close to the origin will actually have non-negligible contributions). The error term might be quite small.

\subsection{Outline.}
\S 3 recalls some of the basic properties of Green's functions and the heat equation and proves the main result. \S 4 derives the new lower bound on the constant, various ways how one could hope to improve it further and discusses the `Meta-Theorem'. \S 5 discusses the representation formulas, Corollary 2 and Corollary 3.

\section{Proof of the Theorem}
Experts in Operator Theory may recognize the main result as a fairly simple consequence of the identity
$$ \frac{1}{\lambda} = \int_{0}^{t} e^{-\lambda s} ds + \frac{e^{-\lambda t}}{\lambda}.$$
However, we will give a completely elementary and explicit argument that tries to be as concrete as possible at all times.
\subsection{Preliminary Facts.} Let $f$ be a measure on a smooth, compact manifold $(M,g)$ without boundary (we will apply this later to $\mathbb{S}^2$ equipped with the canonical metric). We use $e^{t \Delta}f$ to denote the solution of the heat equation
$$ \frac{\partial u}{\partial t} = \Delta u$$
initialized with $u(0, \cdot) = f$ after $t$ units of time. Denoting the $L^2-$normalized spherical harmonics by $-\Delta \phi_k = \lambda_k \phi_k$, we 
have from the linearity of the heat equation and the completeness of $\left( \phi_k \right)_{k=0}^{\infty}$ in $L^2$
 $$ e^{t\Delta}f = \sum_{k = 0}^{\infty}{ e^{-\lambda_k t} \left\langle f, \phi_k \right\rangle \phi_k}.$$
The Green function is defined as the solution of
$$ - \Delta_x \int_{M}{G(x,y) f(y) dy} = f(x).$$
This implies that we have to have
$$ \int_{M} G(x,y) \phi_k(y) dy = \frac{\phi_k(x)}{\lambda_k}$$
and thus, by linearity,
 $$ \int_{M} G(x,y) f(y) dy = \sum_{k=1}^{\infty}{ \frac{ \left\langle f, \phi_k \right\rangle}{\lambda_k}\phi_k(x)}.$$
This is similar to the Sobolev norm $\dot H^{-1}$ which is defined by
 $$ \|f\|_{\dot H^{-1}}^2 = \sum_{k=1}^{\infty}{\frac{ \left\langle f, \phi_k \right\rangle^2}{\lambda_k}}.$$
Note, in particular, that these two notions are related via
\begin{align*}
 \int_{M \times M} G(x,y) f(x) f(y) dx dy &= \left\langle \int_{M} G(x,y) f(y) dy, f(x) \right\rangle \\
 &=  \sum_{k=1}^{\infty}{ \frac{ \left\langle f, \phi_k \right\rangle^2}{\lambda_k}} = \|f\|_{\dot H^{-1}}^2.
 \end{align*}

 We note that the heat equation and the Green function are both spectral multipliers and thus,  whenever $s_1 + t_1 = s_2 + t_2$ and all four numbers are positive,
 \begin{align*}
 \int_{M} \int_{M} G(x,y) e^{s_1\Delta} f(x) e^{t_1 \Delta} g(y) dx dy &= \sum_{k=1}^{\infty}{ e^{-s_1\lambda_k} \frac{\left\langle f, \phi_k\right\rangle \left\langle g, \phi_k \right\rangle}{\lambda_k} e^{-t_1\lambda_k }} \\
  &= \sum_{k=1}^{\infty}{ e^{-s_2\lambda_k} \frac{\left\langle f, \phi_k\right\rangle \left\langle g, \phi_k \right\rangle}{\lambda_k} e^{-t_2\lambda_k }} \\
 &=  \int_{M} \int_{M} G(x,y) e^{s_2\Delta} f(x) e^{t_2 \Delta} g(y) dx dy.
 \end{align*}
We will use this specifically, when $f$ and $g$ are Dirac measures in two distinct points $x_k \neq x_{\ell}$ where this identity implies that
 \begin{align*} \int_{\mathbb{S}^2} \int_{\mathbb{S}^2} G(x,y) e^{t/2\Delta} \delta_{x_k}(x) e^{t/2 \Delta} \delta_{x_{\ell}} (y) dx dy &= 
    \int_{\mathbb{S}^2} \int_{\mathbb{S}^2} G(x,y)  \delta_{x_k}(x) e^{t\Delta} \delta_{x_{\ell}}(y) dx dy \\
    &=  \int_{\mathbb{S}^2} G(x_k,y)  e^{t\Delta} \delta_{x_{\ell}}(y)  dy.
    \end{align*}

\subsection{The identity on $\mathbb{S}^2$.} We will now derive the main result. In \S 3.3 we will explain how the argument can be adapted to other geometries.
We use 
$$ \mathbb{S}^2 = \left\{(x,y,z) \in \mathbb{R}^3: x^2 + y^2 + z^2 = 1\right\}$$
to denote the standard sphere in $\mathbb{R}^3$ equipped with the canonical metric and thus normalized to have surface area $4\pi$. The Green's function in terms of the Euclidean distance is then given by (see e.g. \cite{beltran0})
 $$G(x,y) = \frac{1}{2\pi} \log{ \frac{1}{\|x-y\|}} +c_{2}$$
where $c_{2}$ is chosen so that $G(x, \cdot)$ has mean value 0 on $\mathbb{S}^2$.
Since
 $$\frac{1}{(4\pi)^2} \int_{\mathbb{S}^2} \int_{\mathbb{S}^2} \log \frac{1}{\|x-y\|}dx dy =   \frac{1}{2} - \log{2}$$
 we have from the symmetry of the sphere that for all $x \in \mathbb{S}^2$
  $$\frac{1}{4\pi}  \int_{\mathbb{S}^2} \log \frac{1}{\|x-y\|}dy =   \frac{1}{2} - \log{2}$$
and thus
 $$ \int_{\mathbb{S}^2} \frac{1}{2\pi} \log \frac{1}{\|x-y\|}dx dy =   1 - 2\log{2}$$
from which we deduce that
$$c_2= -\frac{1}{4\pi} + \frac{\log{2}}{2\pi}.$$
We refer to Beltran, Criado del Rey \& Corral \cite{beltran0} for more details and examples of Green functions on other `nice' manifolds.

\begin{proof}[Proof of the Theorem] Recall that
$$ \Delta_y G(x,y) = \frac{1}{4\pi} - \delta_x.$$
Fixing the measure
$$ \mu = \frac{1}{n} \sum_{k=1}^{n}{ \delta_{x_k}},$$
we have
 \begin{align*}
 \left\| e^{(t/2) \Delta} \mu \right\|_{\dot H^{-1}}^2 &= \int_{\mathbb{S}^2} \int_{\mathbb{S}^2} G(x,y) e^{t/2\Delta} \mu(x) e^{t/2 \Delta} \mu (y) dx dy\\
   &= \frac{1}{n^2} \sum_{k, \ell = 1}^n \int_{\mathbb{S}^2} \int_{\mathbb{S}^2} G(x,y) e^{t/2\Delta} \delta_{x_k}(x) e^{t/2 \Delta} \delta_{x_{\ell}} (y) dx dy \\
   &= \frac{1}{n^2} \sum_{k=1}^n \int_{\mathbb{S}^2} \int_{\mathbb{S}^2} G(x,y) e^{t/2\Delta} \delta_{x_k}(x) e^{t/2 \Delta} \delta_{x_{k}} (y) dx dy\\
   &+ \frac{1}{n^2} \sum_{k, \ell = 1 \atop k \neq \ell}^n \int_{\mathbb{S}^2} \int_{\mathbb{S}^2} G(x,y) e^{t/2\Delta} \delta_{x_k}(x) e^{t/2 \Delta} \delta_{x_{\ell}} (y) dx dy.
   \end{align*}
 We first control the off-diagonal terms. As noted above, we can rewrite them as
  \begin{align*} \int_{\mathbb{S}^2} \int_{\mathbb{S}^2} G(x,y) e^{t/2\Delta} \delta_{x_k}(x) e^{t/2 \Delta} \delta_{x_{\ell}} (y) dx dy &= 
    \int_{\mathbb{S}^2} \int_{\mathbb{S}^2} G(x,y)  \delta_{x_k}(x) e^{t\Delta} \delta_{x_{\ell}}(y) dx dy \\
    &=  \int_{\mathbb{S}^2} G(x_k,y)  e^{t\Delta} \delta_{x_{\ell}}(y)  dy.
    \end{align*}
      We also note that, since $x_k \neq x_{\ell}$, as $t \rightarrow 0$,
  $$ \lim_{t \rightarrow \infty}  \int_{\mathbb{S}^2} G(x_k ,y)e^{t \Delta} \delta_{x_{\ell}} (y) dy = G(x_k, x_{\ell}).$$
  Finally, assuming again $x_k \neq x_{\ell}$, we control the variation in time via
  $$ \frac{\partial}{\partial t} e^{t \Delta} \delta_{x_{\ell}} (y) = \Delta_y e^{t \Delta} \delta_{x_{\ell}} (y),$$
  integration by parts and
  $$ \Delta_y G(x_k ,y) =  \frac{1}{4\pi} - \delta_{x_k}$$
  to conclude
 \begin{align*}
 \frac{\partial}{\partial t}  \int_{\mathbb{S}^2} G(x_k ,y)e^{t \Delta} \delta_{x_{\ell}} (y) dy &=   \int_{\mathbb{S}^2} G(x_k ,y) \Delta_y e^{t \Delta} \delta_{x_{\ell}} (y) dy\\
 &=   \int_{\mathbb{S}^2} \Delta_y G(x_k ,y)  e^{t \Delta} \delta_{x_{\ell}} (y) dy \\
 &=  \int_{\mathbb{S}^2} \left(\frac{1}{4\pi} - \delta_{x_k}\right)  e^{t \Delta} \delta_{x_{\ell}} (y) dy\\
 &= \frac{1}{4\pi} - \left( e^{t \Delta} \delta_{x_{\ell}} \right)(x_k).
  \end{align*}
  We use this with the fundamental Theorem of Calculus to conclude that
 \begin{align*}
  \int_{\mathbb{S}^2} \int_{\mathbb{S}^2} G(x,y) e^{t/2\Delta} \delta_{x_k}(x) e^{t/2 \Delta} \delta_{x_{\ell}} (y) dx dy &= 
    \int_{\mathbb{S}^2} \int_{\mathbb{S}^2} G(x,y)  \delta_{x_k}(x) e^{t\Delta} \delta_{x_{\ell}}(y) dx dy \\
    &=    \int_{\mathbb{S}^2} G(x_k,y)  e^{t\Delta} \delta_{x_{\ell}}(y)  dy\\
    &= G(x_k, x_{\ell}) +  \frac{t}{4\pi} - \int_{0}^t \left( e^{s \Delta} \delta_{x_{\ell}} \right)(x_k) ds.
  \end{align*}
  
 It remains to analyze the diagonal terms which are of the form
  \begin{align*}
   \int_{\mathbb{S}^2} G(x_k,y) e^{t \Delta} \delta_{x_{k}} (y) dy.
   \end{align*}
However, by the symmetries of $\mathbb{S}^2$, the value of this integral depends only on $t$ and not on $x_k$.
Collecting all these terms, we have established the identity, for any arbitrary point $z \in \mathbb{S}^2$
 \begin{align*}
\left\| e^{(t/2) \Delta} \mu \right\|_{\dot H^{-1}}^2 &= \frac{1}{n^2} \sum_{k, \ell=1 \atop k \neq \ell}^n G(x_k, x_{\ell}) + \frac{n(n-1)}{n^2} \frac{t}{4\pi} \\
 &- \frac{1}{n^2}\sum_{k, \ell =1 \atop k \neq \ell}^n  \int_{0}^t \left( e^{s \Delta} \delta_{x_{\ell}} \right)(x_k) ds 
   + \frac{1}{n}     \int_{\mathbb{S}^2} G(z,y) e^{t \Delta} \delta_{z} (y) dy.
      \end{align*}
     We recall that
      $$G(x,y) = \frac{1}{2\pi} \log{ \frac{1}{\|x-y\|}} +\left( -\frac{1}{4\pi} + \frac{\log{2}}{2\pi} \right)$$
and therefore arrive at
       \begin{align*}
2 \pi n^2\left\| e^{(t/2) \Delta} \mu \right\|_{\dot H^{-1}}^2 &=   \left( \sum_{k, \ell=1 \atop k \neq \ell}^n \log \frac{1}{\|x_k - x_{\ell}\|} \right)+ 2\pi \left( -\frac{1}{4\pi} + \frac{\log{2}}{2\pi}\right) n(n-1) \\
&+ n(n-1) \frac{t}{2}  - 2\pi \sum_{k, \ell =1 \atop k \neq \ell}^n  \int_{0}^t \left( e^{s \Delta} \delta_{x_{\ell}} \right)(x_k) ds \\
   &+   2\pi n  \int_{\mathbb{S}^2} G(z,y) e^{t \Delta} \delta_{z} (y) dy.
      \end{align*}
Altogether, we arrive at, for any $z \in \mathbb{S}^2$,
\begin{align*}
\sum_{k, \ell =1 \atop k \neq \ell}^n \log \frac{1}{\|x_k - x_{\ell}\|} &= \left(\frac{1}{2} - \log{2} \right)n^2 - \left( \frac{1}{2} - \log{2}\right)n - \frac{n(n-1)t}{2}  \\
& + 2\pi \sum_{k, \ell=1 \atop k \neq \ell}^n  \int_{0}^t \left( e^{s \Delta} \delta_{x_{\ell}} \right)(x_k) ds  -  2\pi n  \int_{\mathbb{S}^2} G(z,y) e^{t \Delta} \delta_{z} (y) dy \\
& +2 \pi n^2\left\| e^{(t/2) \Delta} \mu \right\|_{\dot H^{-1}}^2.
\end{align*}
We see that this implies the desired statement since the expression involving $z \in \mathbb{S}^2$ is completely arbitrary and independent of the points
$\left\{x_1, \dots, x_n\right\}$.
\end{proof}

In particular, we can rearrange the expression and see that the constant (which the theorem guarantees to be independent of the actual points)
$$ X = \frac{1}{2\pi} \sum_{k \neq \ell} \log \frac{1}{\|x_k - x_{\ell}\|}  - \sum_{k \neq \ell}  \int_{0}^t \left( e^{s \Delta} \delta_{x_{\ell}} \right)(x_k) ds - \left\| e^{(t/2) \Delta} \left( \sum_{k=1}^{n} \delta_{x_k} \right) \right\|_{\dot H^{-1}}^2$$
is given by
\begin{align*}
X &= \frac{1}{2\pi} \left(\frac{1}{2} - \log{2} \right)n(n-1)  - \frac{n(n-1)t}{4\pi}  -n  \int_{\mathbb{S}^2} G(z,y) e^{t \Delta} \delta_{z} (y) dy, 
\end{align*}
where $z \in \mathbb{S}^2$ is arbitrary (the integral is constant in $z$).
\subsection{Other manifolds} A similar argument can be carried out on any smooth, compact Riemannian manifold $(M,g)$. We note that we have
$$ \Delta_y G(x,y) = \frac{1}{\mbox{vol}(M)} - \delta_x.$$
The first half of the argument is identical. We arrive at
 \begin{align*}
  \int_{M} \int_{M} G(x,y) e^{t/2\Delta} \delta_{x_k}(x) e^{t/2 \Delta} \delta_{x_{\ell}} (y) dx dy &= G(x_k, x_{\ell}) +  \frac{t}{\mbox{vol}(M)} \\
  &- \int_{0}^t \left( e^{s \Delta} \delta_{x_{\ell}} \right)(x_k) ds.
  \end{align*}
The second half of the argument gives
 \begin{align*}
\left\| e^{(t/2) \Delta} \mu \right\|_{\dot H^{-1}}^2 &= \frac{1}{n^2} \sum_{k \neq \ell} G(x_k, x_{\ell}) + \frac{n(n-1)}{n^2} \frac{t}{4\pi} \\
 &- \frac{1}{n^2}\sum_{k \neq \ell}  \int_{0}^t \left( e^{s \Delta} \delta_{x_{\ell}} \right)(x_k) ds 
   + \frac{1}{n^2}  \sum_{k=1}^{n}     \int_{M} G(x_k,y) e^{t \Delta} \delta_{x_k} (y) dy.
      \end{align*}
We observe that the new quantity       
$$ \frac{1}{n^2}  \sum_{k=1}^{n}     \int_{M} G(x_k,y) e^{t \Delta} \delta_{x_k} (y) dy \qquad \mbox{is}~not~\mbox{independent}$$
of the set of points and will generically depend on them. However, on manifolds with additional symmetry such as $\mathbb{S}^d$, it is possible for each integral
to be actually independent of $x_k$ in which case we obtain a result analogous to our Theorem. This is particularly interesting in other domains for which explicit expression
for the Green function are available:  Beltran, Criado del Rey \& Corral \cite{beltran00} list several. Our argument could conceivably be carried out on all of them (though it is
a priori less clear whether everything can be done in closed form).
On general two-dimensional manifolds, for example, we expect that for time $t$ small, the expression behaves in a way that is indistinguishable from Euclidean space: we expect
$$ G(x_k, y) \sim c \log{\frac{1}{\|x_k -y\|}} + c_2$$
where $c$ does not depend on the point $x_k$. This leads to a logarithmic term (the second term in the asymptotic expansion of $\mathcal{E}_{\log}$) which
is known to only depend on the volume of the manifold (an argument that can be found in Elkies \cite{lang}).\\

\section{A Lower Bound on the Constant: Proof of Corollary }
\subsection{A Lemma}
We need an additional ingredient: an estimate for the quantity that we know is independent of the actual point (but, since we are interested in an explicit constant, needs to be estimated).
\begin{lemma} Let $z \in \mathbb{S}^2$. Then, as $t \rightarrow 0$,
$$ \int_{\mathbb{S}^2} G(z,y) e^{t \Delta} \delta_{z} (y) dy =  - \frac{1}{4\pi} + \frac{\log{2}}{2\pi} + \frac{\gamma}{4\pi} - \frac{\log{(4t)}}{4\pi} + o(1),$$
where $\gamma$ denotes the Euler-Mascheroni constant.
\end{lemma}
\begin{proof} Recalling that
 $$G(x,y) = \frac{1}{2\pi} \log{ \frac{1}{\|x-y\|}} + \left( -\frac{1}{4\pi} + \frac{\log{2}}{2\pi} \right).$$
Noting that the heat kernel of a Dirac measure always has total integral 1 (since the heat kernel preserves the total integral of a function it is being applied to), we have
$$ \int_{\mathbb{S}^2} G(z,y) e^{t \Delta} \delta_{z} (y) dy=  \frac{1}{2\pi} \int_{\mathbb{S}^2} \log{ \left( \frac{1}{\|z-y\|}\right)} e^{t \Delta} \delta_{z} (y) dy +  \left( -\frac{1}{4\pi} + \frac{\log{2}}{2\pi} \right).$$
It remains to analyze the term
$$  \int_{\mathbb{S}^2} \log{ \left( \frac{1}{\|z-y\|}\right)} e^{t \Delta} \delta_{z} (y) dy$$
which is independent of the point $z$. We know that, as $t \rightarrow 0$,
 $$ \left[e^{t\Delta} \delta_{x_k}\right](x) \sim \frac{1}{4 \pi t} \exp \left(-\frac{\|x - x_k\|^2}{4t}\right)$$
and the remaining question is the size of the error term. On $\mathbb{S}^2$, this is rather well understood \cite{barilari, fischer}
and we have for $x,y$ sufficiently close to have uniqueness of geodesics (the formula would be slightly different if $x$ and $y$ were antipodal points but because of the rapid decay of the heat kernel, this does not play a role), $$ \left[e^{t\Delta} \delta_{x_k}\right](x) = \frac{1+ \mathcal{O}(t)}{4 \pi t} \exp \left(-\frac{\|x - x_k\|^2}{4t}\right).$$
 We first note that $e^{t \Delta} \delta_x$ has most of its mass in a $\sqrt{t}-$neighborhood of $x$ and is exponentially decaying outside of that, area distortion is locally quadratic, thus
$$ \int_{\mathbb{S}^2} \log{ \left( \frac{1}{\|z-y\|}\right)} e^{t \Delta} \delta_{z} (y) dy = \left(1 + o\left(\frac{1}{\log{t}}\right)\right) \int_{\mathbb{R}^2}   \log{ \left( \frac{1}{\|y\|}\right)} e^{t \Delta} \delta_{0} (y) dy,$$
where $e^{t\Delta}\delta_0$ denotes the heat kernel in $\mathbb{R}^2$ applied to a Dirac measure in the origin -- for this, there is an explicit formula
 $$ \left[e^{t\Delta_{\mathbb{R}^2}} \delta_{x_k}\right](x) = \frac{1}{4 \pi t} \exp \left(-\frac{\|x - x_k\|^2}{4t}\right).$$

Switching to polar coordinates, we have 
\begin{align*}
 \int_{\mathbb{R}^2}   \log{ \left( \frac{1}{\|y\|}\right)} e^{t \Delta} \delta_{0} (y) dy &=  \int_{\mathbb{R}^2}   \log{ \left( \frac{1}{\|y\|}\right)}  \frac{1}{4 \pi t} \exp \left(-\frac{\|y\|^2}{4t}\right) dy \\
 &= \frac{1}{4\pi t} \int_{0}^{\infty} \log{\left(\frac{1}{r}\right)} \exp\left(-\frac{r^2}{4t} \right) 2 \pi r dr.
\end{align*}
It remains to evaluate this integral: introducing the exponential integral function
$$ \ei(z) = - \int_{-z}^{\infty} \frac{e^{-t}}{t} dt,$$
we have the antiderivative
$$ \int_{}^{} \log{\left(\frac{1}{r}\right)} \exp\left(-\frac{r^2}{4t} \right) 2 \pi r dr = -2\pi t \ei\left(-\frac{r^2}{4t}\right) - 4\pi \exp\left(-\frac{r^2}{4t} \right) t \log{\left(\frac{1}{r}\right)}.$$
This quantity clearly tends to 0 as $r \rightarrow \infty$, it thus remains to understand the behavior as $r \rightarrow 0^+$. We can use the asymptotic expansion for $x \rightarrow 0^+$
$$ \ei(-x) = \log{x} + \gamma - x + \mathcal{O}(x^2)$$
to conclude that, for $r \rightarrow 0^+$,
\begin{align*}
2\pi t \ei\left(-\frac{r^2}{4t}\right) + 4\pi\exp\left(-\frac{r^2}{4t} \right) t \log{\left(\frac{1}{r}\right)} &= 2\pi t \log\left(\frac{r^2}{4t} \right) + 2\pi t \gamma \\
&+ 4\pi t \log{\left(\frac{1}{r}\right)} + \mathcal{O}(r^2)\\
&= 2 \pi t \gamma - 2\pi t \log{(4t)}  + \mathcal{O}(r^2).
\end{align*}
Altogether
\begin{align*}
 \int_{\mathbb{R}^2}   \log{ \left( \frac{1}{\|y\|}\right)} e^{t \Delta} \delta_{0} (y) dy &=  \int_{\mathbb{R}^2}   \log{ \left( \frac{1}{\|y\|}\right)}  \frac{1}{4 \pi t} \exp \left(-\frac{\|y\|^2}{4t}\right) dy \\
 &= \frac{1}{4\pi t} \int_{0}^{\infty} \log{\left(\frac{1}{r}\right)} \exp\left(-\frac{r^2}{4t} \right) 2 \pi r dr
 &= \frac{\gamma}{2} - \frac{\log{(4t)}}{2}.
\end{align*}
\end{proof}

\subsection{Improving the estimate for the constant.}
\begin{proof}[Proof of Corollary 1]
We will again make use of the identity
\begin{align*}
\sum_{k \neq \ell} \log \frac{1}{\|x_k - x_{\ell}\|} &= \left(\frac{1}{2} - \log{2} \right)n^2 + \left( \log{2} - \frac{1}{2}\right)n - \frac{n(n-1)t}{2}  \\
& + 2\pi \sum_{k \neq \ell}  \int_{0}^t \left( e^{s \Delta} \delta_{x_{\ell}} \right)(x_k) ds  -  2\pi n  \int_{\mathbb{S}^2} G(z,y) e^{t \Delta} \delta_{z} (y) dy \\
& +2 \pi n^2\left\| e^{(t/2) \Delta} \mu \right\|_{\dot H^{-1}}^2.
\end{align*}
We will ignore the Sobolev term completely and simply argue that
$$ 2 \pi n^2\left\| e^{(t/2) \Delta} \mu \right\|_{\dot H^{-1}}^2 \geq 0.$$
Likewise, we will ignore the interaction energy: noting that 
$ e^{t\Delta} \delta_{x} (y) \geq 0$
because solutions of the heat equation preserve positivity of the initial datum, we have
$$ 2\pi \sum_{k \neq \ell}  \int_{0}^t \left( e^{s \Delta} \delta_{x_{\ell}} \right)(x_k) ds  \geq 0.$$
Using Lemma 1, we obtain (as long as $t\rightarrow 0$ when $n \rightarrow \infty$)
\begin{align*}
-  2\pi n  \int_{\mathbb{S}^2} G(z,y) e^{t \Delta} \delta_{z} (y) dy &= n\left(\frac{1}{2} - \log{2}\right) - \frac{ n \gamma}{2} + n\frac{\log{(4t)}}{2} + o(n).
\end{align*}
Combining all these ingredients, we obtain the lower bound
\begin{align*}
\sum_{k \neq \ell} \log \frac{1}{\|x_k - x_{\ell}\|} \geq \left(\frac{1}{2} - \log{2} \right)n^2 - \frac{n(n-1)t}{2}  -   \frac{\gamma n}{2}  + n\frac{\log{(4t)}}{2} + o(n).
\end{align*}
Setting $t = 1/n$ leads to
\begin{align*}
\sum_{k \neq \ell} \log \frac{1}{\|x_k - x_{\ell}\|} &\geq \left(\frac{1}{2} - \log{2} \right)n^2 - \frac{1}{2} n \log{n}+  \frac{-\gamma + \log{(4)} -1}{2}n + o(n).
\end{align*}
We note that
$$  \frac{-\gamma + \log{(4)} -1}{2} \sim -0.0954\dots$$
which is the desired result.
\end{proof}

\subsection{How to get further improvements.} At this point one could start wondering how to improve this. Room for improvement comes from our use of the two inequalities
$$ \sum_{k \neq \ell}  \int_{0}^t \left( e^{s \Delta} \delta_{x_{\ell}} \right)(x_k) ds  \geq 0 \qquad \mbox{and} \qquad \left\| e^{(t/2) \Delta} \mu \right\|_{\dot H^{-1}}^2 \geq 0,$$
which are both clearly lossy. However, any sort of serious improvement would require to have at least some knowledge about the minimal energy configuration. We will quickly illustrate this by using the first term: suppose we knew that, at least very locally, the configuration is approximately hexagonal in the sense that the nearest neighbors form a hexagon. This would suggest that each point has $6$ other points at distance (see \S 5.1. for a derivation)
$$ \lambda = \frac{\sqrt{8\pi}}{3^{1/4} \sqrt{n}}.$$
That would then imply that the interaction between two such points $x_k, x_{\ell}$ is at scale
 $$ \left[e^{t\Delta} \delta_{x_{\ell}}\right](x_{k}) \sim \frac{1}{4 \pi t} \exp \left(-\frac{8 \pi}{ \sqrt{3} n \cdot 4 t}\right).$$
Thus we would expect, for $t = 1/n$,
\begin{align*} 2\pi \sum_{k \neq \ell}  \int_{0}^t \left( e^{s \Delta} \delta_{x_{\ell}} \right)(x_k) ds  &\geq 12 \pi n   \int_{0}^{1/n} \frac{1}{4 \pi t} \exp \left(-\frac{8 \pi}{ \sqrt{3} \cdot 4 \cdot n t} \right) dt \\
&\sim 0.017863 n
\end{align*}
which would improve the constant from $c_{\log} \geq -0.095$ to the slightly better one $c_{\log} \geq -0.077$. Incorporating more and more points and picking $t = c/n$ for $c \gg 1$ to shift the importance from the $\dot H^{-1}-$term (which is tricky and global) to the local interaction energy eventually leads to the conjectured results (see \S 5). One could wonder how far one can get by knowing little: some improvement is possible but it is not large. We quickly sketch an argument.

\begin{proposition}
Let $\left\{x_1, \dots, x_n\right\} \subset \mathbb{S}^2$ and let $s > 0$ be arbitrary. Then, as $n \rightarrow \infty$,
$$ \# \left\{ x_i \big| \quad  \exists x_j \neq x_i  ~\emph{such that}~ \|x_i - x_j\| \leq \frac{s}{\sqrt{n}} \right\} \geq \left(1 - \frac{4}{s^2} \frac{2\pi}{\sqrt{3}}\right)n + o(n).$$
\end{proposition}
\begin{proof} Let us denote the number of points that do \textit{not} have this property by $X$. Then we can put a $(s/2)/\sqrt{n}$ ball around each of these points and the balls will not overlap. Since they do not overlap, their maximum density is that of a hexagonal lattice which is $\pi/\sqrt{12}$. Their area is approximately given by the area of the Euclidean counterpart since $n \rightarrow \infty$. Altogether, this means that
$$ X\cdot \frac{s^2 \pi}{ 4n} \leq \frac{\pi}{\sqrt{12}} \cdot 4 \pi.$$
Rearranging gives the desired lower bound on $n-X$.
\end{proof}
We can use this to estimate, as $n\rightarrow \infty$, 
\begin{align*} 2\pi \sum_{k \neq \ell}  \int_{0}^t \left( e^{s \Delta} \delta_{x_{\ell}} \right)(x_k) ds  &\geq  2 \pi  \left(1 - \frac{1}{s^2} \frac{8\pi}{\sqrt{3}}\right) n   \int_{0}^{1/n} \frac{1}{4 \pi t} \exp \left(-\frac{s^2}{ 4 n t} \right) dt.
\end{align*}
This is maximal for $s \sim 4.16$ where it contributes at total of $\sim 0.0002n$ showing the need for more structured information about the minimizing configuration.

\subsection{The `Meta-Theorem'.} At this point, we can explain the idea behind the Meta-Theorem. The arguments in the preceding section explained how one could get slightly more information out of the interaction quantity
$$\sum_{k \neq \ell}  \int_{0}^t \left( e^{s \Delta} \delta_{x_{\ell}} \right)(x_k) ds.$$
Nonetheless, this by itself can never lead to a complete result without analyzing the size of the second term that was dismissed,
$$ 2 \pi \left\| e^{(t/2) \Delta} \sum_{k=1}^{n} \delta_{x_k} \right\|_{\dot H^{-1}}^2 \geq 0.$$
We will now analyze this object. We recall that 
 $$ \left[e^{t\Delta} \delta_{x_k}\right](x) \sim \frac{1}{4 \pi t} \exp \left(-\frac{\|x - x_k\|^2}{4t}\right)$$
and that we are working, approximately, at scale $t = c/n$ for some constant $c>0$.  This means that each one of these terms is essentially a Gaussian at scale
$\sim \sqrt{c/n}$. Without any information about the distribution of the points, it is perhaps not terribly clear what to do with this information. However, if the points are locally arranged in a fairly regular manner, say a hexagonal pattern, then something very nice starts to happen: the Gaussians start to cancel each other out and the arising function will start to behave in a fairly regular manner.

\begin{center}
\begin{figure}[h!]
\begin{tikzpicture}[scale=0.4]
  \foreach \x in {-7,...,7}
    \foreach \y in {-7,...,7}
      {
        \filldraw (\x + 0.5*\y,0.866*\y) circle (0.06cm);
      }
       \draw[ultra thick] (0.3, 0.1) circle (0.5cm);
      \draw[very thick] (0.3, 0.1) circle (1cm);
            \draw[very thick] (0.3, 0.1) circle (1.5cm);
      \draw[thick] (0.3, 0.1) circle (2cm);
            \draw[thick] (0.3, 0.1) circle (2.5cm);
      \draw (0.3, 0.1) circle (3cm);
            \draw (0.3, 0.1) circle (3.5cm);
     \draw[dashed] (0.3, 0.1) circle (4cm);
     \draw[dashed] (0.3, 0.1) circle (4.5cm);
\end{tikzpicture}
\caption{Averaging a Gaussian over a Hexagonal Lattice.}
\label{fig:ave}
\end{figure}
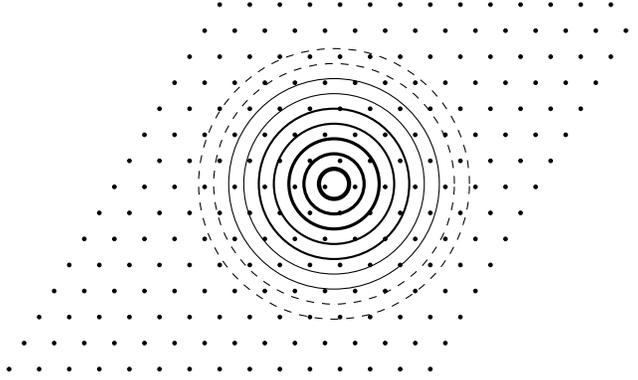
\end{center}

More precisely, see Fig. \ref{fig:ave}, we can interpret the function as averaging over the lattice with a Gaussian weight: the arising function will be close to constant with small fluctuations (the wider the Gaussian, the smaller the fluctuations). Moreover, we see that this function fluctuates at length scale $\sim n^{-1}$ and thus we expect that
$$  \left\| e^{(t/2) \Delta} \sum_{k=1}^{n} \delta_{x_k} \right\|^2_{\dot H^{-1}} \lesssim \frac{1}{n}  \left\| e^{(t/2) \Delta} \sum_{k=1}^{n} \delta_{x_k} \right\|^2_{L^2}.$$
However, we can expand the square and note that self-interactions are small compared to off-diagonal interactions for $c$ large, i.e.
$$ \frac{1}{n}\left\| e^{(t/2) \Delta} \sum_{k=1}^{n} \delta_{x_k} \right\|^2_{L^2} \sim  \frac{1}{n}\sum_{k \neq \ell} \left( e^{t \Delta} \delta_{x_{\ell}} \right)(x_k).$$
However, since $t = c/n$ and $c \gg 1$, we have
$$ \frac{1}{n}\sum_{k \neq \ell} \left( e^{t \Delta} \delta_{x_{\ell}} \right)(x_k) \ll  \int_0^t \sum_{k \neq \ell} \left( e^{s \Delta} \delta_{x_{\ell}} \right)(x_k) ds.$$
This is the desired Meta-Theorem.  We emphasize that this argument does not require the points to be exactly arranged like a hexagonal lattice, it suffices if they
are approximately arranged in a regular shape. Naturally, the stronger the information that one has about the geometry of the minimizing configuration, the easier it is to make this part of the argument rigorous.

\section{A Formula for {\Large $c_{\log}$}: Corollary 2 and 3}
In this section, we will aim to use our approach to identify the correct linear constant assuming that the minimizing configuration 
behaves locally like the hexagonal lattice. After that, we will combine various results to establish Corollary 3.

\subsection{Proof of Corollary 2.} We have to understand the behavior of the expression
$$  - \frac{n(n-1)t}{2}  + 2\pi \sum_{k \neq \ell}  \int_{0}^t \left( e^{s \Delta} \delta_{x_{\ell}} \right)(x_k) ds \qquad \mbox{for}~t = \frac{c}{n}$$
and $c$ being arbitrarily large (though we will assume it to be quite small compared to $n$). Since $t \ll 1$, the heat kernel will be localized to some $\sim c^{1/2} n^{-1/2}$ neighborhood
around a point and we can approximate this by assuming that everything happens in $\mathbb{R}^2$ and that the local configuration of points behaves like the
hexagonal lattice.
In $\mathbb{R}^2$, we have an explicit expression for the heat kernel
$$ \left[e^{s\Delta} \delta_{x_k}\right](x) = \frac{1}{4 \pi s} \exp \left(-\frac{\|x - x_k\|^2}{4s}\right).$$
In particular, this allows us to incorporate the self-interaction for any $s>0$ via
$$ \sum_{k \neq \ell}  \left( e^{s \Delta} \delta_{x_{\ell}} \right)(x_k)  = - \frac{n}{4\pi s} + \sum_{k, \ell}   \left( e^{s \Delta} \delta_{x_{\ell}} \right)(x_k).$$

\begin{center}
\begin{figure}[h!]
\begin{tikzpicture}[scale=2]
\filldraw (0,0) circle (0.04cm);
\filldraw (1,0) circle (0.04cm);
\filldraw (0.5,0.866) circle (0.04cm);
\draw [thick] (0,0) -- (1,0) -- (0.5, 0.866) -- (0,0);
\filldraw (-1,0) circle (0.04cm);
\filldraw (-0.5,0.866) circle (0.04cm);
\filldraw (1,0) circle (0.04cm);
\filldraw (-0.5,-0.866) circle (0.04cm);
\filldraw (0.5,-0.866) circle (0.04cm);
\draw [thick] (0.5,0.866) -- (-0.5,0.866) -- (0,0) -- (-1,0) -- (-0.5, -0.866) -- (0.5, -0.866) -- (1,0);
\draw [thick] (-0.5, 0.866) -- (-1,0);
\draw [thick] (-0.5, -0.866) -- (0,0) -- (0.5, -0.866);
\node at (0.5, 0.1) {$\lambda$};
\end{tikzpicture}
\caption{Local scaling of the hexagonal lattice.}
\label{fig:hex}
\end{figure}
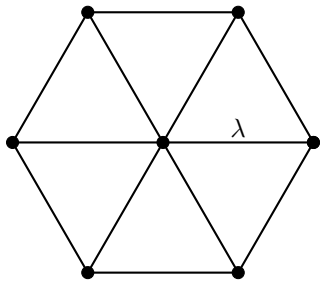
\end{center}

In the next step, we determine the local density of the hexagonal lattice: each lattice point is part of six triangles while each triangle is comprised of three points. Thus there are $2n$ triangles in total covering an area of $|\mathbb{S}^2|  = 4\pi$. Invoking the area of an equilateral triangle (see Fig. \ref{fig:hex}), we have
$$ \frac{\sqrt{3}}{4} \lambda^2 = \frac{4\pi}{2n} \qquad \mbox{and thus} \qquad \lambda = \frac{\sqrt{8\pi}}{3^{1/4} \sqrt{n}}.$$
Let us use $\Lambda$ to denote the standard hexagonal lattice in $\mathbb{R}^2$ generated by the vectors
$$ v_1 = (1,0) \qquad \mbox{and} \qquad v_2 = \left(\frac{1}{2}, \frac{\sqrt{3}}{2}\right).$$
It remains to understand the behavior of
$$ \frac{1}{4\pi s}\sum_{k_1, k_2 \in \mathbb{Z}} \exp \left( - \frac{\lambda^2}{4s} \left\| k_1 v_1 + k_2 v_2 \right\|^2\right) =  \frac{1}{4\pi s}\sum_{k_1, k_2 \in \mathbb{Z}} \exp \left( - \lambda^2\frac{k_1^2 + k_1 k_2 + k_2^2}{4s} \right) .$$
Introducing the real number $0 < q < 1$ by 
$$ q = \exp\left(-\frac{\lambda^2}{4s}\right),$$
we can write this expression as
$$  \frac{1}{4\pi s} \sum_{k_1, k_2 \in \mathbb{Z}} \exp \left( - \lambda^2\frac{k_1^2 + k_1 k_2 + k_2^2}{4s} \right) =  \frac{1}{4\pi s} \sum_{k_1, k_2 \in \mathbb{Z}} q^{k_1^2 + k_1 k_2 + k_2^2}.$$
However, this sum is the cubic analogue of Jacobi theta functions (see e.g. Borwein \& Borwein \cite{borwein} and Faulhuber \cite{markus}). We introduce
$$ L(q) =  \sum_{k_1, k_2 \in \mathbb{Z}} q^{k_1^2 + k_1 k_2 + k_2^2}.$$
Thus
$$   \frac{1}{4\pi s} \sum_{k_1, k_2 \in \mathbb{Z}} q^{k_1^2 + k_1 k_2 + k_2^2} = \frac{L(q)}{4\pi s}.$$
Summarizing, for all $s>0$, in the limit as local structures converge to the hexagonal lattice within a $\sqrt{c}/\sqrt{n}-$window around most points,
\begin{align*}
 \sum_{k \neq \ell}  \left( e^{s \Delta} \delta_{x_{\ell}} \right)(x_k)  &= - \frac{n}{4\pi s} + \sum_{k, \ell}   \left( e^{s \Delta} \delta_{x_{\ell}} \right)(x_k) \\
 &= - \frac{n}{4\pi s} + \frac{n}{4\pi s}\cdot L\left(\exp\left(-\frac{\lambda^2}{4s}\right) \right)\\
 &=  - \frac{n}{4\pi s} + \frac{n}{4\pi s} \cdot L\left(\exp\left(-\frac{2\pi}{\sqrt{3} n s}\right) \right)\\
 &=    \frac{n}{4\pi s} \left[ L\left(\exp\left(-\frac{2\pi}{\sqrt{3} n s}\right) \right) - 1\right].
 \end{align*}
Setting $t = c/n$, we have (ignoring terms smaller than $\sim n$)
\begin{align*}
 - \frac{n(n-1)t}{2}  + 2\pi \sum_{k \neq \ell}  \int_{0}^{c/n} \left( e^{s \Delta} \delta_{x_{\ell}} \right)(x_k) ds &=  - \frac{n c}{2}+ o(n) \\
 &+ n\int_{0}^{c/n}  \frac{1}{2s} \left(  L\left(\exp\left(-\frac{2\pi}{\sqrt{3} n s}\right) \right) - 1 \right)ds.
\end{align*}
Changing variables leads to
$$n\int_{0}^{c/n}  \frac{1}{2s} \left(  L\left(\exp\left(-\frac{2\pi}{\sqrt{3} n s}\right) \right) - 1 \right)ds = n\int_{0}^{c}  \frac{1}{2s} \left(  L\left(\exp\left(-\frac{2\pi}{\sqrt{3} s}\right) \right) - 1 \right)ds.$$
We recall from Lemma 1 that
\begin{align*}
 -  2\pi n  \int_{\mathbb{S}^2} G(z,y) e^{t \Delta} \delta_{z} (y) dy &= n\left(\frac{1}{2} - \log{2}\right) - \frac{\gamma}{2} n - \frac{n}{2} \log{n} + \frac{n}{2} \log{(4c)} + o(n).
\end{align*}
Collecting all these things, we have established that for any $c > 0$ and for configurations of points that behave locally like the hexagonal lattices in most parts of $\mathbb{S}^2$ with a small underlying error
\begin{align*}
\sum_{k \neq \ell} \log \frac{1}{\|x_k - x_{\ell}\|} &= \left(\frac{1}{2} - \log{2} \right)n^2 - \frac{n}{2} \log{n} - \frac{n c}{2}\\
 &+ n\int_{0}^{c}  \frac{1}{2s} \left(  L\left(\exp\left(-\frac{2\pi}{\sqrt{3} s}\right) \right) - 1 \right)ds\\
& - \frac{\gamma}{2} n + \frac{n}{2} \log{(4c)}  +2 \pi n^2\left\| e^{(c/2n) \Delta} \mu \right\|_{\dot H^{-1}}^2 + o(n).
\end{align*}
We see, following the discussion in \S 4.4, that for any fixed $c>0$, the quantity $2 \pi n^2\left\| e^{(c/2n) \Delta} \mu \right\|_{\dot H^{-1}}^2$ will scale
linearly in $n$ but with a constant that goes to 0 as $c \rightarrow \infty$ (and somewhat rapidly as well). Therefore, letting $c \rightarrow \infty$, we see
that the linear term is given by
$$ c_{1} = \lim_{c \rightarrow \infty} - \frac{c}{2} +\int_{0}^{c}  \frac{1}{2s} \left(  L\left(e^{-\frac{2\pi}{\sqrt{3} s}} \right) - 1 \right)ds - \frac{\gamma}{2} + \frac{\log{(4c)}}{2}.$$
The constant $c_{\tiny \mbox{BHS}}$ \cite{brauchart} has been obtained as the leading order term under the assumption of a locally hexagonal lattice and since our computation
is based on exactly the same assumption, we have $c_{1} = c_{\tiny \mbox{BHS}}$.
If it could be shown that the minimizing configuration is indeed locally hexagonal (see B\'etermin \& Sandier \cite{betermin} and Sandier \& Serfaty \cite{sandier}), this would tell us that $c_{\log} = c_{\tiny \mbox{BHS}}$ and thus $c_1 = c_{\log}$.

\subsection{Proof of Corollary 3} Returning to the integral representation for $c_{\log}$ derived above, we remark that we can write
$$ L\left(\exp\left(-\frac{2\pi}{\sqrt{3} s}\right) \right) - 1 = \sum_{(k_1, k_2) \in \mathbb{Z}^2 \atop (k_1, k_2) \neq (0,0)} \exp\left( -\frac{2\pi}{\sqrt{3} s} (k_1^2+k_1 k_2 + k_2^2)\right).$$
Observe that integrating a single such term leads to the incomplete gamma function 
$$ \int_{0}^{c} \frac{1}{2s}\exp\left( - \frac{\alpha}{s} \right) = \frac{1}{2} \Gamma\left(0, \frac{\alpha}{c} \right).$$
Therefore
$$ \int_{0}^{c}  \frac{1}{2s} \left(  L\left(\exp\left(-\frac{2\pi}{\sqrt{3} s}\right) \right) - 1 \right)ds = \frac{1}{2} \sum_{  (k_1, k_2) \in \mathbb{Z}^2 \atop (k_1, k_2) \neq (0,0)}  \Gamma\left(0,  \frac{2\pi}{\sqrt{3}c} (k_1^2 + k_1 k_2 + k_2^2) \right).$$
Using 
$$ c_{\log} = \lim_{c \rightarrow \infty} - \frac{c}{2} +\int_{0}^{c}  \frac{1}{2s} \left(  L\left(e^{-\frac{2\pi}{\sqrt{3} s}} \right) - 1 \right)ds - \frac{\gamma}{2} + \frac{\log{(4c)}}{2}$$
as well as (this assumption is satisfied since we do only work with the hexagonal lattice here)
$$  c_{\log} =  2\log{2} + \frac{1}{2} \log\frac{2}{3} + 3 \log \frac{\sqrt{\pi}}{\Gamma(1/3)},$$
we arrive at, for $c \rightarrow \infty$,
\begin{align*}  \sum_{  (k_1, k_2) \in \mathbb{Z}^2 \atop (k_1, k_2) \neq (0,0)}  \Gamma\left(0,  \frac{2\pi}{\sqrt{3}c} (k_1^2 + k_1 k_2 + k_2^2) \right) &= c - \log{(4c)} + \gamma + 4\log{2} \\&+ \log{\frac{2}{3}} + 6 \log{\frac{\sqrt{\pi}}{\Gamma(1/3)}} + o(1).
\end{align*}
A change of variables, $\varepsilon = 2\pi/(\sqrt{3} c)$, results in the desired statement.\\

\textbf{Acknowledgment.} The author is grateful to Laurent B\'etermin and Carlos Beltran for helpful discussions.

\end{document}